\documentclass[11pt,leqno]{amsart}

\usepackage{amsmath}
\usepackage{amssymb}
\usepackage{a4wide}
\usepackage{bbm}
\usepackage{graphics}
\usepackage{epsfig}
\usepackage{color}
\usepackage{mathrsfs}
\usepackage[sans]{dsfont}

\newtheorem{theorem}{Theorem}[section]
\newtheorem{lemma}[theorem]{Lemma}

\newcounter{hypo}
\newcounter{hypoa}

\newcounter{hypoaa}
\newcounter{hypobb}

\def\C{{\mathbb C}}

\def\N{{\mathbb N}} 
\def\R{{\mathbb R}} 

\def\Z{{\mathbb Z}}

\def\CD{\mathcal {D}}

\def\CH{\mathcal {H}}
\def\CI{{\mathcal I}}

\def\CM{\mathcal {M}}

\def\CO{\mathcal {O}}
\def\CP{\mathcal {P}}
\def\CQ{\mathcal {Q}}

\def\CW{\mathcal {W}}

\def\SH{\mathscr {H}}
\def\SI{\mathscr {I}}

\def\re{\mathop{\rm Re}\nolimits}
 \def\im{\mathop{\rm Im}\nolimits}

\def\Op{\mathop{\rm Op}\nolimits}

\newcommand{\tr}{\operatorname{tr}}

\newcommand{\eqt}{\operatorname{ess-qt}}
\newcommand{\qt}{\operatorname{qt}}

\newcommand{\supp}{\operatorname{supp}}
\newcommand{\spe}{\operatorname{sp}}

\def\dist{\mathop{\rm dist}\nolimits}
\def\diag{\mathop{\rm diag}\nolimits}

\def\<{\langle}
\def\>{\rangle}

\def\res{\mathop{\rm Res}\nolimits}

\newcommand{\fract}[2]{\genfrac{}{}{0pt}{}{\scriptstyle #1}{\scriptstyle #2}}

\makeatletter
 \@addtoreset{equation}{section}
 \makeatother
 
\title{An example of resonance instability}

\author[J.-F. Bony]{Jean-Fran\c{c}ois Bony}
\address{Jean-Fran\c{c}ois Bony, IMB, CNRS (UMR 5251), Universit\'e de Bordeaux, 33405 Talence, France}
\email{bony@math.u-bordeaux.fr}
\author[S. Fujii\'e]{Setsuro Fujii\'e}
\address{Setsuro Fujii\'e, Department of Mathematical Sciences, Ritsumeikan University, 1-1-1 Noji-Higashi, Kusatsu, 525-8577 Japan}
\email{fujiie@fc.ritsumei.ac.jp}
\author[T. Ramond]{Thierry Ramond}
\address{Thierry Ramond, Universit\'e Paris-Saclay, CNRS, Laboratoire de math\'ematiques d'Orsay, 91405, Orsay, France}
\email{thierry.ramond@universite-paris-saclay.fr}
\author[M. Zerzeri]{Maher Zerzeri}
\address{Maher Zerzeri, Universit\'e Sorbonne Paris-Nord, LAGA, CNRS (UMR 7539), 93430 Villetaneuse, France}
\email{zerzeri@math.univ-paris13.fr}


\thanks{\textbf{Acknowledgments:} The second author was partially supported by the JSPS KAKENHI Grant 18K03384}

\begin{document}

\begin{abstract}
We construct a semiclassical Schr\"{o}dinger operator such that the imaginary part of its resonances closest to the real axis changes by a term of size $h$ when a real compactly supported potential of size $o ( h )$ is added.
\end{abstract}

\maketitle

\section{Introduction} \label{s1}

In this note, we consider semiclassical Schr\"{o}dinger operators $P$ on $L^{2} ( \R^{n} )$, $n \geq 1$,
\begin{equation}\label{a3}
P = - h^{2} \Delta + V ( x ),
\end{equation}
where $V \in C^{\infty}_{0} ( \R^{n} ; \R )$ is a real-valued smooth compactly supported potential. Depending on the situation, one may also work with such operators outside a compact smooth obstacle with Dirichlet boundary condition. Since $P$ is a compactly supported perturbation of $- h^{2} \Delta$, the resonances of $P$ near the real axis are well-defined through the analytic distortion method or using the meromorphic extension of its truncated resolvent. We send back the reader to the books of Sj\"{o}strand \cite{Sj07_01} or Dyatlov and Zworski \cite{DyZw16_01} for a general presentation of resonance theory, and we denote $\res ( P )$ the set of resonances of $P$.

The stability of the resonances is a rather touchy question. Indeed, we do not know yet whether the concept of resonance persists under the perturbation by a non-analytic non-exponentially decreasing potential. Therefore, we only consider here perturbations of Schr\"{o}dinger operators \eqref{a3} by subprincipal real-valued smooth compactly supported potentials of the form $h^{\tau} W ( x )$ with $\tau > 0$ and $W \in C^{\infty}_{0} ( \R^{n} ; \R )$. But even in this setting, the stability of resonances is a subtle problem since stability results and instability results can be obtained for the same operator.

On one hand, the resonances tend to be stable as other spectral objects like the eigenvalues. This is particularly clear when the resonances are defined by complex distortion, since the usual perturbation theory of discrete spectrum can be directly applied to the distorted operator. But, even if the resonances are defined as the poles of the meromorphic extension of some weighted resolvent, Agmon \cite{Ag98_01,Ag98_02} has proved their stability. On the other hand, the resonances can be unstable since they do not come from a self-adjoint problem. Thus, some typical non self-adjoint effects may occur concerning the resonances even if $P$ is self-adjoint. For instance, the distorted operator may have a Jordan block or the truncated resolvent may have a pole of algebraic order greater than $1$ (see e.g. Sj\"{o}strand \cite[Section 4]{Sj87_01}).

\begin{figure}
\begin{center}
\begin{picture}(0,0)%
\includegraphics{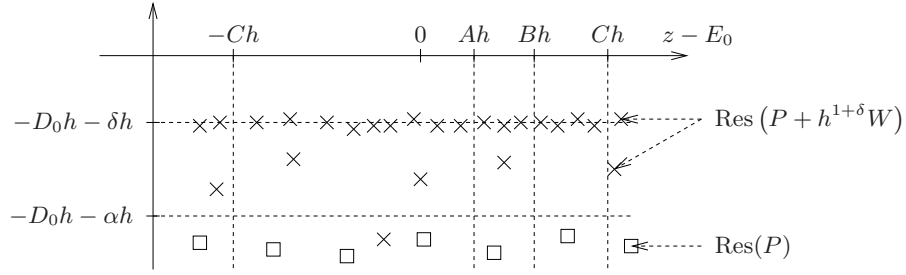}%
\end{picture}%
\setlength{\unitlength}{1105sp}%
\begingroup\makeatletter\ifx\SetFigFont\undefined%
\gdef\SetFigFont#1#2#3#4#5{%
  \reset@font\fontsize{#1}{#2pt}%
  \fontfamily{#3}\fontseries{#4}\fontshape{#5}%
  \selectfont}%
\fi\endgroup%
\begin{picture}(13837,6044)(-1221,-383)
\put(-449,689){\makebox(0,0)[rb]{\smash{{\SetFigFont{9}{10.8}{\rmdefault}{\mddefault}{\updefault}$- D_{0} h - \alpha h$}}}}
\put(7201,4739){\makebox(0,0)[b]{\smash{{\SetFigFont{9}{10.8}{\rmdefault}{\mddefault}{\updefault}$A h$}}}}
\put(8551,4739){\makebox(0,0)[b]{\smash{{\SetFigFont{9}{10.8}{\rmdefault}{\mddefault}{\updefault}$B h$}}}}
\put(-449,2789){\makebox(0,0)[rb]{\smash{{\SetFigFont{9}{10.8}{\rmdefault}{\mddefault}{\updefault}$- D_{0} h - \delta h$}}}}
\put(1801,4739){\makebox(0,0)[b]{\smash{{\SetFigFont{9}{10.8}{\rmdefault}{\mddefault}{\updefault}$- C h$}}}}
\put(10201,4739){\makebox(0,0)[b]{\smash{{\SetFigFont{9}{10.8}{\rmdefault}{\mddefault}{\updefault}$C h$}}}}
\put(6001,4739){\makebox(0,0)[b]{\smash{{\SetFigFont{9}{10.8}{\rmdefault}{\mddefault}{\updefault}$0$}}}}
\put(12226,4739){\makebox(0,0)[b]{\smash{{\SetFigFont{9}{10.8}{\rmdefault}{\mddefault}{\updefault}$z - E_{0}$}}}}
\put(12601,2864){\makebox(0,0)[lb]{\smash{{\SetFigFont{9}{10.8}{\rmdefault}{\mddefault}{\updefault}$\res \big( P + h^{1 +  \delta} W \big)$}}}}
\put(12601, 14){\makebox(0,0)[lb]{\smash{{\SetFigFont{9}{10.8}{\rmdefault}{\mddefault}{\updefault}$\res ( P )$}}}}
\end{picture}%
\end{center}
\caption{The spectral setting of Theorem \ref{a1}.} \label{f1}
\end{figure}

Our instability result is the following.

\begin{theorem}[Resonance instability]\sl \label{a1}
In dimension $n = 2$, one can construct an operator $P$ and a potential $W$ as above satisfying the following property for all $\delta > 0$ small enough. There exist a set $\SH \subset ] 0 , 1 ]$ with $0 \in \overline{\SH}$ and constants $D_{0} , E_{0} , \alpha > 0$ such that, for all $C > 0$ and $- C \leq A < B \leq C$,

$i)$ On one hand, $P$ has no resonance $z$ with $\re z \in E_{0} + [ - C h , C h ]$ and
\begin{equation}
\im z \geq - D_{0} h - \alpha h ,
\end{equation}
for $h \in \SH$ small enough.

$ii)$ On the other hand, the resonances $z$ of $P + h^{1 + \delta} W$ with $\re z \in E_{0} + [ A h , B h ]$ closest to the real axis satisfy
\begin{equation} \label{a2}
\im z \sim - D_{0} h - \delta h ,
\end{equation}
for $h \in \SH$ small enough.
\end{theorem}

The result is illustrated in Figure \ref{f1}. Theorem \ref{a1} $ii)$ provides at least one resonance $z$ of $P + h^{1 + \delta} W$ satisfying $\re z \in E_{0} + [ A h , B h ]$ and $\im z \sim - D_{0} h - \delta h$. But its proof shows that the number of such resonances is at least of order $\vert \ln h \vert$. In particular, the {\sl essential quantum trapping} in $E_{0} + [ A h , B h ]$ defined by
\begin{equation}
\eqt ( Q ) = \lim_{n \to + \infty} \limsup_{\fract{h \to 0}{h \in \SH}} \inf_{\fract{z_{1} , \ldots , z_{n} \in\res ( Q )}{\re z_{\bullet} \in E_{0} + [ A h , B h ]}} \sup_{\fract{z \in\res ( Q ) \setminus \{ z_{1} , \ldots , z_{n} \}}{\re z \in E_{0} + [ A h , B h ]}} \frac{h}{\vert \im z \vert} ,
\end{equation}
increases by at least $( \alpha - \delta ) ( D_{0} + \alpha )^{- 1} ( D_{0} + \delta )^{- 1}$ when we add the perturbation $h^{1 + \delta} W$ to the operator $P$. Thus, the resonance instability described here is not an anomaly due to an exceptional resonance or a Jordan block but a phenomenon mixing geometry and analysis.

In the statement of the previous result, we do not specify the subset of semiclassical parameters $\SH$. In fact, depending on the geometric situation, the resonance instability may occur on the whole interval $\SH = ] 0 , 1 ]$ or only near a sequence $\SH$ like $\{ j^{- 1} ; \ j \in \N^{*} \}$. Operators corresponding to these different situations are given at the end of Section \ref{s2}.

For $0 < \kappa \ll 1$ fixed, one can show that the resonances $z$ of $P + \kappa h W$ with $\re z \in E_{0} + [ A h , B h ]$ closest to the real axis satisfy $\im z \sim - D_{0} h$ for $h \in \SH$ small enough. The proof of this point is similar to that of Theorem \ref{a1}. On the contrary, for larger values of $\kappa$, some cancellations may appear and $P + \kappa h W$ may have a resonance free region of size $D_{0} h + \alpha h$ below the real axis as for $P$.

The constructions in Theorem \ref{a1} can be realized in any dimension $n \geq 2$, but our method of proof does not work in dimension $n = 1$. Indeed, the Hamiltonian vector field must have an anisotropic hyperbolic fixed point. Nevertheless, we do not know yet if the resonance instability phenomenon described here occurs in dimension one.

Let $P_{\theta}$ denote the operator $P$ in the proof of Theorem \ref{a1} after a complex distortion of angle $\theta = h \vert \ln h \vert$. Its resolvent satisfies a polynomial estimate in $\Omega = E_{0} + [ - C h , C h ] + i [ - D_{0} h - \alpha h , h ]$. This means that, for some $M > 0$, we have
\begin{equation}
\big\Vert ( P_{\theta} - z )^{- 1} \big\Vert \lesssim h^{- M} ,
\end{equation}
uniformly for $z \in \Omega$. By the usual perturbation argument, it implies that $P + Q$ has no resonance in $\Omega$ for any distortable perturbation $Q$ of size $o ( h^{M} )$. The stability of resonances under small enough perturbations has already been observed (see e.g. Agmon \cite{Ag98_01,Ag98_02}). Summing up, the resonances of $P$ are stable for perturbations of size $o ( h^{M} )$ and unstable for some perturbations of size $h^{1 + \delta}$ (showing that $M \geq 1 + \delta$).

\begin{figure}
\begin{center}
\begin{picture}(0,0)%
\includegraphics{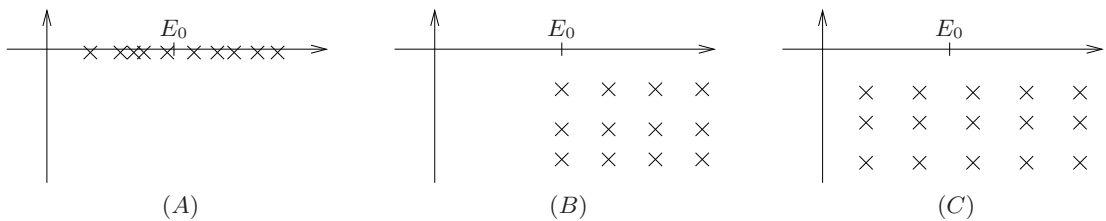}%
\end{picture}%
\setlength{\unitlength}{1105sp}%
\begingroup\makeatletter\ifx\SetFigFont\undefined%
\gdef\SetFigFont#1#2#3#4#5{%
  \reset@font\fontsize{#1}{#2pt}%
  \fontfamily{#3}\fontseries{#4}\fontshape{#5}%
  \selectfont}%
\fi\endgroup%
\begin{picture}(24644,4661)(-9621,700)
\put(11551,4739){\makebox(0,0)[b]{\smash{{\SetFigFont{9}{10.8}{\rmdefault}{\mddefault}{\updefault}$E_{0}$}}}}
\put(11701,764){\makebox(0,0)[b]{\smash{{\SetFigFont{9}{10.8}{\rmdefault}{\mddefault}{\updefault}$(C)$}}}}
\put(2851,4739){\makebox(0,0)[b]{\smash{{\SetFigFont{9}{10.8}{\rmdefault}{\mddefault}{\updefault}$E_{0}$}}}}
\put(-5849,4739){\makebox(0,0)[b]{\smash{{\SetFigFont{9}{10.8}{\rmdefault}{\mddefault}{\updefault}$E_{0}$}}}}
\put(3001,764){\makebox(0,0)[b]{\smash{{\SetFigFont{9}{10.8}{\rmdefault}{\mddefault}{\updefault}$(B)$}}}}
\put(-5699,764){\makebox(0,0)[b]{\smash{{\SetFigFont{9}{10.8}{\rmdefault}{\mddefault}{\updefault}$(A)$}}}}
\end{picture}%
\end{center}
\caption{The resonances generated by $(A)$ a well in the island, $(B)$ a non-degenerate critical point and $(C)$ a hyperbolic closed trajectory.} \label{f2}
\end{figure}

The present result is obtained for a Schr\"{o}dinger operator whose trapped set at energy $E_{0}$ consists of a hyperbolic fixed point and homoclinic trajectories, following our recent paper \cite{BoFuRaZe18_01}. In fact, the instability phenomenon obtained here does not hold in the geometric settings previously studied (see Figure \ref{f2}). In the ``well in the island'' situation, the resonances are known to be exponentially close to the real axis (see Helffer and Sj\"{o}strand \cite{HeSj86_01} for globally analytic potentials and Lahmar-Benbernou, Martinez and the second author \cite{FuLaMa11_01} for potentials analytic at infinity). Adding a subprincipal real potential $h W ( x )$ does not change this properties. When the trapped set at energy $E_{0}$ consists of a non-degenerate critical point (say at $( x_{0} , 0 ) \in T^{*} \R^{n}$), Sj\"{o}strand \cite{Sj87_01} has proved that the resonances form, modulo $o ( h )$, a quarter of a rectangular lattice which is translated by $h W ( x_{0} )$ when a subprincipal potential $h W ( x )$ is added. Finally, the asymptotic of the resonances generated by a hyperbolic closed trajectory has been obtained by G{\'e}rard and Sj\"{o}strand \cite{GeSj87_01} (see also Ikawa \cite{Ik83_01} and G{\'e}rard \cite{Ge88_01} for obstacles). Modulo $o ( h )$, they form half of a rectangular lattice which is translated by a real quantity after perturbation by a real potential $h W ( x )$. Summing up, the imaginary part of the resonances is {\sl very stable} in the three previous examples: it moves only by $o ( h )$ when a perturbation by a real potential of size $h$ is applied. In other words, if the {\sl quantum trapping} (or maximum of the quantum lifetime) in $E_{0} + [ -C h , C h ]$ of an operator $Q$ is defined by
\begin{equation}
\qt ( Q ) = \limsup_{h \to 0} \sup_{\fract{z \in\res ( Q )}{\re z \in E_{0} + [ -C h , C h ]}} \frac{h}{\vert \im z \vert} ,
\end{equation}
with the conventions that $\qt ( Q ) = + \infty$ if the limit diverges and $\qt ( Q ) = 0$ if $Q$ has no resonance, we have $\qt ( P ) = \qt ( P + h W )$ in these examples. The situation is completely opposite in Theorem \ref{a1} since a self-adjoint perturbation of size $o ( h )$ induces a change of size $1$ of the quantum trapping. By definition, we always have $\qt ( Q ) \in [ 0 , + \infty ]$ and $\qt ( Q ) \geq \eqt ( Q )$. Moreover, if the resonance expansion of the quantum propagator holds, we have $\Vert \chi e^{- i t Q / h} \varphi ( Q ) \chi \Vert \approx e^{t / \qt ( Q )}$ for $t \gg1$ and $h$ in an appropriate sequence, justifying the name of quantum trapping. Other results in scattering theory provide resonance free regions, that is upper bounds on the quantum trapping, under geometric assumptions. In general, the bounds obtained do not depend on the subprincipal symbol, assumed to be self-adjoint in an appropriate class (see for instance Section 3.2 of Nonnenmacher and Zworski \cite{NoZw09_01}). In the present setting, Section 3.1 of \cite{BoFuRaZe18_01} implies $\qt ( P ) \leq D_{0}^{- 1}$, but Theorem \ref{a1} $i)$ shows that this inequality is not sharp.

Theorem \ref{a1} may seem natural since the distorted resolvent is generally large in the unphysical sheet and small perturbations may produce eigenvalues. More precisely, the norm of the distorted resolvent is known to be larger than $h^{- 1}$, that is
\begin{equation*}
\big\Vert ( P_{\theta} - z )^{- 1} \big\Vert \gg h^{- 1} ,
\end{equation*}
with $\im z < 0$, in many cases (see e.g. Burq and two of the authors \cite{BoBuRa10_01} or Dyatlov and Waters \cite{DyWa16_01}). By the pseudospectral theory (see e.g. Section I.4 of Trefethen and Embree \cite{TrEm05_01}), there exists a bounded operator $Q_{\theta}$ of size $o ( h )$ such that $z$ is precisely an eigenvalue of $P_{\theta} + Q_{\theta}$. Nevertheless, it is not clear that $Q_{\theta}$ is the distortion of some operator $Q$, that $Q$ is a potential and that $Q$ is self-adjoint. In fact, as explained in the previous paragraph, this is not always the case.

This instability phenomenon is due to the non self-adjoint nature of the resonances (even for self-adjoint operators). Such a property never holds for the usual spectrum in the self-adjoint framework. Indeed, for any self-adjoint operator $P$ and any bounded perturbation $W$, the spectrum of $P + W$ satisfies
\begin{equation*}
\sigma ( P + W ) \subset \sigma ( P ) + B ( 0 , \Vert W \Vert ) .
\end{equation*}
Thus, a perturbation of size $h^{1 + \delta}$ of a self-adjoint operator can not lead to a perturbation of size $h$ of its spectrum.

The operator $P$ and the potential $W$ are constructed in Section \ref{s2}. The instability phenomenon stated in Theorem \ref{a1} is proved in Section \ref{s3}.

\section{Construction of the operators} \label{s2}

\begin{figure}
\begin{center}
\begin{picture}(0,0)%
\includegraphics{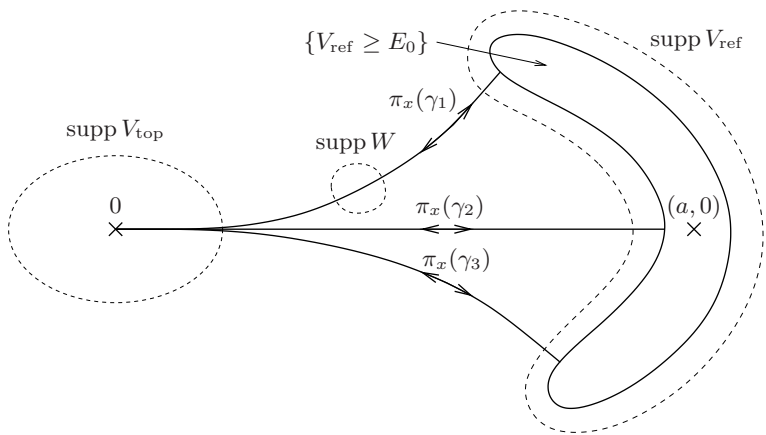}%
\end{picture}%
\setlength{\unitlength}{1105sp}%
\begingroup\makeatletter\ifx\SetFigFont\undefined%
\gdef\SetFigFont#1#2#3#4#5{%
  \reset@font\fontsize{#1}{#2pt}%
  \fontfamily{#3}\fontseries{#4}\fontshape{#5}%
  \selectfont}%
\fi\endgroup%
\begin{picture}(16960,9507)(-10814,-9744)
\put(-8399,-4786){\makebox(0,0)[b]{\smash{{\SetFigFont{9}{10.8}{\rmdefault}{\mddefault}{\updefault}$0$}}}}
\put(-8399,-3061){\makebox(0,0)[b]{\smash{{\SetFigFont{9}{10.8}{\rmdefault}{\mddefault}{\updefault}$\supp V_{\rm top}$}}}}
\put(-899,-4786){\makebox(0,0)[b]{\smash{{\SetFigFont{9}{10.8}{\rmdefault}{\mddefault}{\updefault}$\pi_{x} ( \gamma_{2} )$}}}}
\put(-749,-5986){\makebox(0,0)[b]{\smash{{\SetFigFont{9}{10.8}{\rmdefault}{\mddefault}{\updefault}$\pi_{x} ( \gamma_{3} )$}}}}
\put(4651,-1036){\makebox(0,0)[b]{\smash{{\SetFigFont{9}{10.8}{\rmdefault}{\mddefault}{\updefault}$\supp V_{\rm ref}$}}}}
\put(-2999,-3361){\makebox(0,0)[b]{\smash{{\SetFigFont{9}{10.8}{\rmdefault}{\mddefault}{\updefault}$\supp W$}}}}
\put(-2774,-1111){\makebox(0,0)[b]{\smash{{\SetFigFont{9}{10.8}{\rmdefault}{\mddefault}{\updefault}$\{ V_{\rm ref} \geq E_{0} \}$}}}}
\put(-1499,-2386){\makebox(0,0)[b]{\smash{{\SetFigFont{9}{10.8}{\rmdefault}{\mddefault}{\updefault}$\pi_{x} ( \gamma_{1} )$}}}}
\put(4576,-4786){\makebox(0,0)[b]{\smash{{\SetFigFont{9}{10.8}{\rmdefault}{\mddefault}{\updefault}$( a , 0 )$}}}}
\end{picture}%
\end{center}
\caption{The potentials $V = V _{\rm top} + V_{\rm ref}$ and $W$.} \label{f3}
\end{figure}

To construct a Schr\"{o}dinger operator $P = - h^{2} \Delta + V ( x )$ as in \eqref{a3} with unstable resonances, we follow Example 4.23 and Example 4.24 (B) of \cite{BoFuRaZe18_01}. We send back the reader to this paper for a slightly different presentation, some close geometric situations and general results about resonances generated by homoclinic trajectories. As usual, $p ( x , \xi ) = \xi^{2} + V ( x )$ denotes the symbol of $P$, its associated Hamiltonian vector field is
\begin{equation*}
H_{p} = \partial_{\xi} p \cdot \partial_{x} - \partial_{x} p \cdot \partial_{\xi} = 2 \xi \cdot \partial_{x} - \nabla V ( x ) \cdot \partial_{\xi} ,
\end{equation*}
and the trapped set at energy $E$ for $P$ is
\begin{equation*}
K ( E ) = \big\{ ( x , \xi ) \in p^{- 1} ( E ) ; \ t \mapsto \exp ( t H_{p} ) ( x , \xi ) \text{ is bounded} \big\} .
\end{equation*}
Recall that $K ( E )$ is compact and stable by the Hamiltonian flow for $E > 0$.

In dimension $n =2$, we consider the potential
\begin{equation}
V ( x ) = V_{\rm top} ( x ) + V_{\rm ref} ( x ) ,
\end{equation}
as in Figure \ref{f3} and described below. On one hand, the potential $V_{\rm top}$ is of the form $V_{\rm top} ( x ) = V_{1} ( x_{1} ) V_{2} ( x_{2} )$ where the functions $V_{\bullet} \in C^{\infty}_{0} ( \R )$ are single barriers (see Figure \ref{f4}) with
\begin{equation*}
V_{1} ( x_{1} ) = E_{0} - \frac{\lambda_{1}^{2}}{4} x_{1}^{2} + \CO ( x_{1}^{3} )  \qquad \text{and} \qquad   V_{2} ( x_{2} ) = 1 - \frac{\lambda_{2}^{2}}{4 E_{0}} x_{2}^{2} + \CO ( x_{2}^{3} ) ,
\end{equation*}
near $0$ and $0 < \lambda_{1} < \lambda_{2}$. In particular, $V_{\rm top}$ is an anisotropic bump,
\begin{equation*}
V_{\rm top} ( x ) = E_{0} - \frac{\lambda_{1}^{2}}{4} x_{1}^{2} - \frac{\lambda_{2}^{2}}{4} x_{2}^{2} + \CO ( x^{3} ) ,
\end{equation*}
near $0$ and $( 0 , 0 )$ is a hyperbolic fixed point for $H_{p}$. The stable/unstable manifold theorem ensures the existence of the incoming/outgoing Lagrangian manifolds $\Lambda_{\pm}$ characterized by
\begin{equation*}
\Lambda_{\pm} = \big\{ ( x , \xi ) \in T^{*} \R^{2} ; \ \exp ( t H_{p} ) ( x , \xi ) \to ( 0 , 0 ) \text{ as } t \to \mp \infty \big\} .
\end{equation*}
They are stable by the Hamiltonian flow and included in $p^{- 1} ( E_{0} )$. Eventually, there exist two smooth functions $\varphi_{\pm}$, defined in a vicinity of $0$, satisfying
\begin{equation} \label{a4}
\varphi_{\pm} ( x ) = \pm \sum_{j = 1}^{2} \frac{\lambda_{j}}{4} x_{j}^{2} + \CO ( x^{3} ) ,
\end{equation}
and such that $\Lambda_{\pm} = \{ ( x , \xi ) ; \ \xi = \nabla \varphi_{\pm} (x) \}$ near 
$( 0 , 0 )$.

\begin{figure}[!h]
\begin{center}
\begin{picture}(0,0)%
\includegraphics{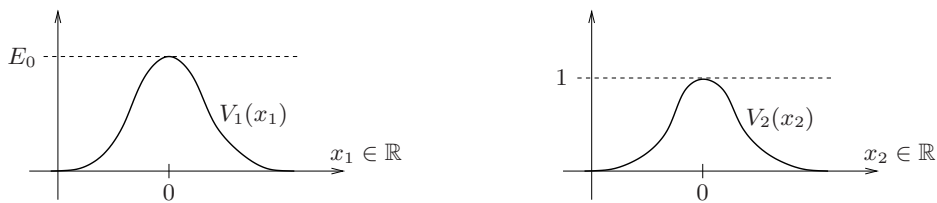}%
\end{picture}%
\setlength{\unitlength}{1184sp}%
\begingroup\makeatletter\ifx\SetFigFont\undefined%
\gdef\SetFigFont#1#2#3#4#5{%
  \reset@font\fontsize{#1}{#2pt}%
  \fontfamily{#3}\fontseries{#4}\fontshape{#5}%
  \selectfont}%
\fi\endgroup%
\begin{picture}(18262,4052)(136,-5816)
\put(14701,-5761){\makebox(0,0)[b]{\smash{{\SetFigFont{9}{10.8}{\rmdefault}{\mddefault}{\updefault}$0$}}}}
\put(3526,-5761){\makebox(0,0)[b]{\smash{{\SetFigFont{9}{10.8}{\rmdefault}{\mddefault}{\updefault}$0$}}}}
\put(4576,-4111){\makebox(0,0)[lb]{\smash{{\SetFigFont{9}{10.8}{\rmdefault}{\mddefault}{\updefault}$V_{1} ( x_{1} )$}}}}
\put(6901,-4936){\makebox(0,0)[lb]{\smash{{\SetFigFont{9}{10.8}{\rmdefault}{\mddefault}{\updefault}$x_{1} \in \R$}}}}
\put(18076,-4936){\makebox(0,0)[lb]{\smash{{\SetFigFont{9}{10.8}{\rmdefault}{\mddefault}{\updefault}$x_{2} \in \R$}}}}
\put(11626,-3361){\makebox(0,0)[lb]{\smash{{\SetFigFont{9}{10.8}{\rmdefault}{\mddefault}{\updefault}$1$}}}}
\put(15601,-4186){\makebox(0,0)[lb]{\smash{{\SetFigFont{9}{10.8}{\rmdefault}{\mddefault}{\updefault}$V_{2} ( x_{2} )$}}}}
\put(151,-2911){\makebox(0,0)[lb]{\smash{{\SetFigFont{9}{10.8}{\rmdefault}{\mddefault}{\updefault}$E_{0}$}}}}
\end{picture}%
\end{center}
\caption{The potentials $V_{1}$ and $V_{2}$.} \label{f4}
\end{figure}

On the other hand, the reflecting potential $V_{\rm ref}$ is non-trapping and localized near $( a , 0 ) \in \R^{2}$ with $a$ large enough. A dynamical result (Lemma B.1 of \cite{BoFuRaZe18_01}) ensures that no Hamiltonian trajectory of energy $E_{0}$ can start from the support of $V_{\rm ref}$, touch the support of $V_{\rm top}$ and then come back to the support of $V_{\rm ref}$. Thus, a trapped trajectory of energy $E_{0}$ is either $\{ ( 0 , 0 ) \}$ or a Hamiltonian trajectory starting asymptotically from the origin, touching the support of $V_{\rm ref}$ and coming back to the origin; these latter trajectories are called homoclinic. In other words, $K ( E_{0} )$ satisfies
\begin{equation*}
K ( E_{0} ) = \Lambda_{-} \cap \Lambda_{+} ,
\end{equation*}
and $\CH = \Lambda_{-} \cap \Lambda_{+} \setminus \{ ( 0 , 0 ) \}$ denotes the set of homoclinic trajectories.

Giving to $V_{\rm ref}$ the form of a ``croissant'' barrier, we can make sure that $\CH$ consists of a finite number of trajectories $\{ \gamma_{1} , \ldots , \gamma_{K} \}$ on which $\Lambda_{-}$ and $\Lambda_{+}$ intersect transversally. In the sequel, we will need at least two homoclinic trajectories, that is $K \geq 2$. Such a geometric configuration can also be realized replacing the potential barrier $V_{\rm ref}$ by an obstacle $\CO$ having essentially the form of $\{ V_{\rm ref} ( x ) > E_{0} \}$, the operator being $P = - h^{2} \Delta_{\R^{2} \setminus \CO} + V_{\rm top} ( x )$ with Dirichlet boundary condition (see Figure \ref{f5}). In that case, one can easily realize a situation where $K = 3$ whereas it seems complicated to have $K = 2$ (see Example 4.14 of \cite{BoFuRaZe18_01}).

\begin{figure}
\begin{center}
\begin{picture}(0,0)%
\includegraphics{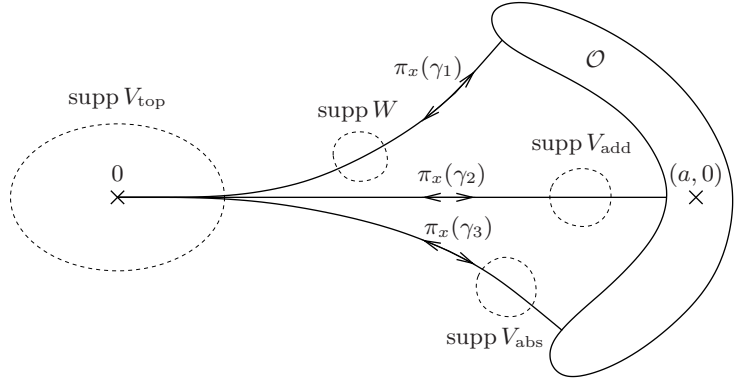}%
\end{picture}%
\setlength{\unitlength}{1105sp}%
\begingroup\makeatletter\ifx\SetFigFont\undefined%
\gdef\SetFigFont#1#2#3#4#5{%
  \reset@font\fontsize{#1}{#2pt}%
  \fontfamily{#3}\fontseries{#4}\fontshape{#5}%
  \selectfont}%
\fi\endgroup%
\begin{picture}(16246,8461)(-10814,-9217)
\put(2026,-4111){\makebox(0,0)[b]{\smash{{\SetFigFont{9}{10.8}{\rmdefault}{\mddefault}{\updefault}$\supp V_{\rm add}$}}}}
\put(-8399,-4786){\makebox(0,0)[b]{\smash{{\SetFigFont{9}{10.8}{\rmdefault}{\mddefault}{\updefault}$0$}}}}
\put(2326,-2236){\makebox(0,0)[b]{\smash{{\SetFigFont{9}{10.8}{\rmdefault}{\mddefault}{\updefault}$\CO$}}}}
\put(-8399,-3061){\makebox(0,0)[b]{\smash{{\SetFigFont{9}{10.8}{\rmdefault}{\mddefault}{\updefault}$\supp V_{\rm top}$}}}}
\put(-899,-4786){\makebox(0,0)[b]{\smash{{\SetFigFont{9}{10.8}{\rmdefault}{\mddefault}{\updefault}$\pi_{x} ( \gamma_{2} )$}}}}
\put(-749,-5986){\makebox(0,0)[b]{\smash{{\SetFigFont{9}{10.8}{\rmdefault}{\mddefault}{\updefault}$\pi_{x} ( \gamma_{3} )$}}}}
\put(-1424,-2386){\makebox(0,0)[b]{\smash{{\SetFigFont{9}{10.8}{\rmdefault}{\mddefault}{\updefault}$\pi_{x} ( \gamma_{1} )$}}}}
\put(-2999,-3361){\makebox(0,0)[b]{\smash{{\SetFigFont{9}{10.8}{\rmdefault}{\mddefault}{\updefault}$\supp W$}}}}
\put( 76,-8461){\makebox(0,0)[b]{\smash{{\SetFigFont{9}{10.8}{\rmdefault}{\mddefault}{\updefault}$\supp V_{\rm abs}$}}}}
\put(4576,-4786){\makebox(0,0)[b]{\smash{{\SetFigFont{9}{10.8}{\rmdefault}{\mddefault}{\updefault}$( a , 0 )$}}}}
\end{picture}%
\end{center}
\caption{A realization with a potential $V_{\rm top}$ and an obstacle $\CO$.} \label{f5}
\end{figure}

As explained in the next section (see also Example 4.24 of \cite{BoFuRaZe18_01}), the resonance instability is governed by the function
\begin{equation} \label{a5}
\mu ( \tau , h ) = \Gamma \Big( \frac{\lambda_{1} + \lambda_{2}}{2 \lambda_{1}} - i \frac{\tau}{\lambda_{1}} \Big) e^{- \frac{\pi \tau}{2 \lambda_{1}}} \sum_{k = 1}^{K} e^{i A_{k} / h} B_{k} e^{i T_{k} \tau} ,
\end{equation}
for $\tau \in \C$, where $A_{k} , B_{k} , T_{k}$ are dynamical quantities related to the curve $\gamma_{k} = ( x_{k} , \xi_{k} )$. We recall quickly how these quantities are defined and send back the reader to \cite{BoFuRaZe18_01} for the proof of convergence of the various objects. First,
\begin{equation*}
A_{k} = \int_{\gamma_{k}} \xi \cdot d x ,
\end{equation*}
is the action along $\gamma_{k}$. Let $\nu_{k}$ denote the Maslov index of $\Lambda_{+}$ along $\gamma_{k}$. The function $x_{k} ( t )$ has the following asymptotics
\begin{equation*}
x_{k} (t) = g_{\pm}^{k} e^{\pm \lambda_{1} t} + o \big( e^{\pm \lambda_{1} t} \big) ,
\end{equation*}
as $t \to \mp \infty$ for some vector $g_{\pm}^{k} \in \R^{2}$. As a matter of fact, $g_{\pm}^{k}$ is collinear to the first vector of the canonical basis $( 1 , 0 )$ and do not vanish. Eventually, if $\gamma_{k} ( t , y ) = ( x_{k} ( t , y) , \xi_{k} ( t , y ) ) : \R \times \R \longrightarrow T^{*} \R^{2}$ is a smooth parametrization of $\Lambda_{+}$ by Hamiltonian curves such that $\gamma_{k} ( t , 0 ) = \gamma_{k} ( t )$, the limits
\begin{equation*}
\begin{aligned}
\CM_{k}^{+} & = \lim_{s \to - \infty} \sqrt{\Big\vert \det \frac{\partial x_{k} ( t , y )}{\partial ( t , y )} \vert_{t = s , \ y = 0} \Big\vert} e^{- s \frac{\lambda_{1} + \lambda_{2}}{2}} ,   \\
\CM_{k}^{-} & = \lim_{s \to + \infty} \sqrt{\Big\vert \det \frac{\partial x_{k} ( t , y )}{\partial ( t , y )} \vert_{t = s , \ y = 0} \Big\vert} e^{- s \frac{\lambda_{2} - \lambda_{1}}{2}} ,
\end{aligned}
\end{equation*}
exist and belong to $] 0 , + \infty [$. With these notations,
\begin{equation}
\begin{aligned}
B_{k} &= \sqrt{\frac{\lambda_{1}}{2 \pi}} \frac{\CM_{k}^{+}}{\CM_{k}^{-}} e^{- \frac{\pi}{2} ( \nu_{k} + \frac{1}{2} ) i} \big\vert g_{-}^{k} \big\vert \big( i \lambda_{1} \vert g_{+}^{k} \vert \vert g_{-}^{k} \vert \big)^{- \frac{\lambda_{1} + \lambda_{2}}{2 \lambda_{1}}} , \\
T_{k} &= \frac{\ln ( \lambda_{1} \vert g_{+}^{k} \vert \vert g_{-}^{k} \vert )}{\lambda_{1}} .
\end{aligned}
\end{equation}
Note that $B_{k} \in \C \setminus \{ 0 \}$ and $T_{k} \in \R$.

The idea is to find a geometric situation and a set $\SH \subset ] 0 , 1 ]$ with $0 \in \overline{\SH}$ such that
\begin{equation} \label{a6}
\mu ( \tau , h ) = 0 ,
\end{equation}
for all $\tau \in \C$ and $h \in \SH$. For simplicity, we take in the sequel $K = 3$ as in Figure \ref{f3} or \ref{f5} and assume that the trajectories $\gamma_{1}$ and $\gamma_{3}$ are symmetric. In particular, $A_{1} = A_{3}$, $B_{1} = B_{3}$ and $T_{1} = T_{3}$. We consider two situations:

\underline{Case (I):} $A_{1} \neq A_{2}$ (say $A_{2} > A_{1}$), $2 B_{1} = B_{2} e^{i \nu}$, $\nu \in \R$, and $T_{1} = T_{2}$. Using \eqref{a5} and the symmetry of $\gamma_{1}$ and $\gamma_{3}$, these relations imply that \eqref{a6} holds true with
\begin{equation*}
\SH = \Big\{ \frac{A_{2} - A_{1}}{( 2 j + 1 ) \pi + \nu} ; \ j \in \N \Big\} .
\end{equation*}
The required relations can be realized since $T_{2}$ is only given by the potential $V$ on the line $\R \times \{ 0 \}$ if $\partial_{x_{2}} V ( x_{1} , 0 ) = 0$ for all $x_{1} \in \R$, whereas $B_{2}$ is given by $\partial_{x_{2}}^{2} V$ on $\R \times \{ 0 \}$. If $V_{\rm ref}$ is replaced by an obstacle $\CO$, one may need an additional potential $V_{\rm add}$ in order to satisfy these relations (see Figure \ref{f5}).

\underline{Case (II):} $A_{1} = A_{2}$, $2 B_{1} = - B_{2}$ and $T_{1} = T_{2}$. In this setting, \eqref{a6} holds true with $\SH = ] 0 , 1 ]$. These relations can been obtained as before. More precisely, one can adjust $V_{\rm ref}$ on $\R \times \{ 0 \}$ with $\partial_{x_{2}} V = 0$ on $\R \times \{ 0 \}$ in order to have $A_{1} = A_{2}$ and $T_{1} = T_{2}$. Then, playing on the Maslov index and on $\partial_{x_{2}}^{2} V_{\rm ref}$ on $\R \times \{ 0 \}$, one can obtain $2 B_{1} = - B_{2}$.

Adding an absorbing potential $- i h \vert \ln h \vert V_{\rm abs}$, with $V_{\rm abs} \geq 0$, it is possible to artificially remove a homoclinic trajectory and thus to work with only $K = 2$ trajectories (see Remark 2.1 $ii)$ and Example 4.14 of \cite{BoFuRaZe18_01}). The resulting operator will be non self-adjoint (dissipative) but the conclusions of Theorem \ref{a1} will still hold.

For the perturbation $W$, we take any non-negative $C^{\infty}_{0} ( \R^{2} ; \R )$ function supported away from the support of $V$ and non-zero on the base space projection of only one homoclinic trajectory. In the sequel, we will assume that this trajectory is $\gamma_{1}$ as in Figures \ref{f3} and \ref{f5}. We assume that $W = 0$ near the support of $V$ only to simplify the discussion. The same way, $W \geq 0$ and $W$ non-zero on $\pi_{x} ( \gamma_{1} )$ can be weakened to $\int_{\R} W ( x_{1} ( t ) ) \, d t \neq 0$. Finally, $W = c_{1} W_{1} + c_{2} W_{2} + c_{3} W_{3}$, with $W_{j}$ non-zero only on $\gamma_{j}$, may be suitable for Theorem \ref{a1} generically with respect to $c_{j} \in \R$.

\section{Proof of the spectral instability} \label{s3}

We consider the operators constructed in the previous section. In particular, we work in dimension $n = 2$, the trapped set of energy $E_{0} > 0$ has $K = 3$ homoclinic trajectories and \eqref{a6} holds true for $h \in\SH$. Following Chapter 4 of \cite{BoFuRaZe18_01}, the resonances of $P$ closest to the real axis are given by the $3 \times 3$ matrix $\CQ$ whose entries are
\begin{equation} \label{a7}
\CQ_{k , \ell} ( z , h ) = e^{i A_{k} / h} \Gamma \big( S ( z , h ) / \lambda_{1} \big) \sqrt{\frac{\lambda_{1}}{2 \pi}} \frac{\CM_{k}^{+}}{\CM_{k}^{-}} e^{- \frac{\pi}{2} ( \nu_{k} + \frac{1}{2} ) i} \big\vert g_{-}^{\ell} \big\vert \big( i \lambda_{1} \vert g_{+}^{k} \vert \vert g_{-}^{\ell} \vert \big)^{- S ( z , h ) / \lambda_{1}} ,
\end{equation}
with rescaled spectral parameter
\begin{equation} \label{d5}
S ( z , h ) = \frac{\lambda_{1} + \lambda_{2}}{2} - i \frac{z - E_{0}}{h} .
\end{equation}
The same way, the entries of the corresponding matrix for $\widetilde{P} = P + h^{1 + \delta} W$ are
\begin{equation} \label{a8}
\widetilde{\CQ}_{k , \ell} ( z , h ) = \left\{
\begin{aligned}
&e^{- i w h^{\delta}} \CQ_{k , \ell} ( z , h )  &&\text{ if } k = 1 , \\
&\CQ_{k , \ell} ( z , h )  &&\text{ if } k \neq 1 ,
\end{aligned} \right.
\end{equation}
with the notation $w = \int_{\R} W ( x_{1} ( t ) ) \, d t \neq 0$.

\begin{lemma}\sl \label{a16}
The matrices $\CQ$ and $\widetilde{\CQ}$ are of rank one with non-zero entries. Moreover, $\CQ^{2} ( z , h ) = 0$ for all $z \in \C$ and $h \in \SH$.
\end{lemma}

\begin{proof}
Since $\CM^{\pm}_{\bullet} \neq 0$ and $g_{\pm}^{\bullet} \neq 0$, the entries of $\CQ$ and $\widetilde{\CQ}$ are always non-zero. From \eqref{a7}, the entries of $\CQ$ can be written $\CQ_{k , \ell} = \alpha_{k} \beta_{\ell}$ for some $\alpha_{k} , \beta_{k} \in \C \setminus \{ 0 \}$. In particular, $\CQ = \alpha ( \beta , \cdot )$ with $\alpha , \beta \in \C^{3} \setminus \{ 0 \}$ and $\CQ$ is of rank one (the same thing for $\widetilde{\CQ}$). Thus, $0$ is an eigenvalue of $\CQ$ of multiplicity at least $2$ and the last eigenvalue is given by its trace, that is
\begin{align*}
\tr ( \CQ ) &= \Gamma \big( S ( z , h ) / \lambda_{1} \big) \sum_{k = 1}^{3} e^{i A_{k} / h} \sqrt{\frac{\lambda_{1}}{2 \pi}} \frac{\CM_{k}^{+}}{\CM_{k}^{-}} e^{- \frac{\pi}{2} ( \nu_{k} + \frac{1}{2} ) i} \big\vert g_{-}^{k} \big\vert \big( i \lambda_{1} \vert g_{+}^{k} \vert \vert g_{-}^{k} \vert \big)^{- \frac{\lambda_{1} + \lambda_{2}}{2 \lambda_{1}} + i \frac{z - E_{0}}{\lambda_{1} h}}  \\
&= \mu \Big( \frac{z - E_{0}}{h} , h \Big) .
\end{align*}
For $h \in \SH$, all the eigenvalues of $\CQ$ are zero from \eqref{a6} and $\CQ$ is nilpotent. Since $\CQ^{j} = \alpha ( \beta , \alpha )^{j - 1} ( \beta , \cdot )$, $\CQ$ is nilpotent iff $( \beta , \alpha ) = 0$ iff $\CQ^{2} = 0$.
\end{proof}

Let $\CW$ be the $3 \times 3$ diagonal matrix $\CW = \diag ( - i w , 0 , 0 )$. The eigenvalues of $\CW \CQ ( z , h )$ are
\begin{equation} \label{a9}
- i w \CQ_{1 , 1} ( z , h ) , 0 , 0 .
\end{equation}
In the present setting, the quantization rule for $\widetilde{P}$ takes the following form: we say that $z$ is a pseudo-resonance of $\widetilde{P}$ when
\begin{equation} \label{a10}
1 \in \spe \big( h^{S ( z , h ) / \lambda_{1} - 1 / 2 + \delta} \CW \CQ ( z , h ) \big) .
\end{equation}
The set of pseudo-resonances is denoted by $\res_{0} ( \widetilde{P} )$. Since \eqref{a10} is similar to Definition 4.2 of \cite{BoFuRaZe18_01}, we can adapt Proposition 4.3 and Lemma 11.3 of \cite{BoFuRaZe18_01} in our case and obtain the following asymptotic of the pseudo-resonances.

\begin{lemma}\sl \label{a11}
Let $0 < \delta < 1 / 2$, $C , \beta > 0$ and $\varepsilon ( h )$ be a function which goes to $0$ as $h \to 0$. Then, uniformly for $\tau \in [ - C , C ]$, the pseudo-resonances $z$ of $\widetilde{P}$ in
\begin{equation} \label{a12}
E_{0} + [ - C h , C h ] + i \Big[ - \Big( \frac{\lambda_{2}}{2} + \delta \lambda_{1} \Big) h - C \frac{h}{\vert \ln h \vert} , h \Big] ,
\end{equation}
with $\re z \in E_{0} + \tau h + h \varepsilon ( h ) [ - 1 , 1 ]$ satisfy $z = z_{q} ( \tau ) + o ( h \vert \ln h \vert^{- 1} )$ with
\begin{equation}  \label{a13}
z_{q} ( \tau ) = E_{0} - \frac{A_{1} \lambda_{1}}{\vert \ln h \vert} + 2 q \pi \lambda_{1} \frac{h}{\vert \ln h \vert} - i h \Big( \frac{\lambda_{2}}{2} + \delta \lambda_{1} \Big) + i \ln ( \widetilde{\mu} ( \tau ) ) \lambda_{1} \frac{h}{\vert \ln h \vert} ,
\end{equation}
and
\begin{equation*}
\widetilde{\mu} ( \tau ) = w \Gamma \Big( \frac{1}{2} - \delta - i \frac{\tau}{\lambda_{1}} \Big) \sqrt{\frac{\lambda_{1}}{2 \pi}} \frac{\CM_{1}^{+}}{\CM_{1}^{-}} e^{- \frac{\pi}{2} ( \nu_{1} + \frac{3}{2} ) i} \big\vert g_{-}^{1} \big\vert \big( i \lambda_{1} \vert g_{+}^{1} \vert \vert g_{-}^{1} \vert \big)^{- \frac{1}{2} + \delta+ i \frac{\tau}{\lambda_{1}}} ,
\end{equation*}
for some $q \in \Z$. On the other hand, for each $\tau \in [ - C , C ]$ and $q \in \Z$ such that $z_{q} ( \tau )$ belongs to \eqref{a12} with a real part lying in $E_{0} + \tau h + h \varepsilon ( h ) [ - 1 , 1 ]$, there exists a pseudo-resonance $z$ satisfying $z = z_{q} ( \tau ) + o ( h \vert \ln h \vert^{- 1} )$ uniformly with respect to $q , \tau$. Moreover, there exists $M > 0$ such that, for all $z \in \eqref{a12}$, we have
\begin{equation} \label{a14}
\dist \big( z , \res_{0} ( \widetilde{P} ) \big) > \beta \frac{h}{\vert \ln h \vert} \quad \Longrightarrow \quad \big\Vert \big( 1 - h^{S / \lambda_{1} - 1 / 2 + \delta} \CW \CQ \big)^{- 1} \big\Vert \leq M .
\end{equation}
\end{lemma}

In the lemma, we have used that the eigenvalues of $\CW \CQ$ are explicitly given by \eqref{a9} and that two of them are zero. On the contrary, note that $\widetilde{\mu} ( \tau )$ is a smooth function which does not vanish and that there are a lot of pseudo-resonances in \eqref{a12}. The assumption $0 < \delta < 1 / 2$ allows to avoid the poles of the $\Gamma$ function. This result implies the following resolvent estimates at the classical level.

\begin{lemma}\sl \label{a15}
For all $0 < \delta < 1 / 2$, $\nu = \lambda_{1} / 4 + \lambda_{2} / 2$ and $C , \beta > 0$, the following properties are satisfied for $h \in \CH$ small enough.

$i)$ For all $z \in E_{0} + [ - C h , C h ] + i [ - \nu h , h ]$, we have
\begin{equation} \label{a17}
\big\Vert \big( 1 - h^{S / \lambda_{1} - 1 / 2} \CQ \big)^{- 1} \big\Vert \lesssim \max \big( 1 , h^{\frac{\lambda_{2}}{2 \lambda_{1}} + \frac{\im z}{\lambda_{1} h}} \big) .
\end{equation}

$ii)$ For all $z \in \eqref{a12}$ with $\dist (z , \res_{0} ( \widetilde{P} ) ) > \beta h \vert \ln h \vert^{- 1}$, we have
\begin{equation} \label{a18}
\big\Vert \big( 1 - h^{S / \lambda_{1} - 1 / 2} \widetilde{\CQ} \big)^{- 1} \big\Vert \lesssim h^{- \delta} .
\end{equation}
\end{lemma}

The particular value of $\nu$ in Lemma \ref{a15} has no particular meaning. We only need $\nu > D_{0} = \lambda_{2} / 2$ for Theorem \ref{a1} and $\nu < \lambda_{1} / 2 + \lambda_{2} / 2$ to avoid the poles of $\Gamma$.

\begin{proof}
Since $\CQ^{2} = 0$ by Lemma \ref{a16}, we get
\begin{equation} \label{a32}
\big( 1 - h^{S / \lambda_{1} - 1 / 2} \CQ \big)^{- 1} = 1 + h^{S / \lambda_{1} - 1 / 2} \CQ .
\end{equation}
Using that $\vert h^{S / \lambda_{1} - 1 / 2} \vert = h^{\frac{\lambda_{2}}{2 \lambda_{1}} + \frac{\im z}{\lambda_{1} h}}$ and that $\CQ ( z , h )$ is uniformly bounded for $z \in E_{0} + [ - C h , C h ] + i [ - \nu h , h ]$, this identity yields \eqref{a17}.

On the other hand, \eqref{a8}, $e^{- i w h^{\delta}} = 1 - i w h^{\delta} + \CO ( h^{2 \delta} )$ and $\CQ^{2} = 0$ give
\begin{align}
1 - h^{S / \lambda_{1} - 1 / 2} \widetilde{\CQ} &= 1 - h^{S / \lambda_{1} - 1 / 2} \CQ - h^{S / \lambda_{1} - 1 / 2 + \delta} \CW \CQ + \CO ( h^{\frac{\lambda_{2}}{2 \lambda_{1}} + \frac{\im z}{\lambda_{1} h} + 2 \delta} ) \CQ   \nonumber \\
&= \Big( 1 - h^{S / \lambda_{1} - 1 / 2 + \delta} \CW \CQ + \CO ( h^{\frac{\lambda_{2}}{2 \lambda_{1}} + \frac{\im z}{\lambda_{1} h} + 2 \delta} ) \CQ \Big) \big( 1 - h^{S / \lambda_{1} - 1 / 2} \CQ \big)  \nonumber \\
&= \Big( 1 + \CO ( h^{\frac{\lambda_{2}}{2 \lambda_{1}} + \frac{\im z}{\lambda_{1} h} + 2 \delta} ) \big( 1 - h^{S / \lambda_{1} - 1 / 2 + \delta} \CW \CQ \big)^{- 1} \Big) \nonumber \\
&\qquad \qquad \qquad \qquad \big( 1 - h^{S / \lambda_{1} - 1 / 2 + \delta} \CW \CQ \big) \big( 1 - h^{S / \lambda_{1} - 1 / 2} \CQ \big) .   \label{a19}
\end{align}
We have $\vert h^{S / \lambda_{1} - 1 / 2} \vert = h^{\frac{\lambda_{2}}{2 \lambda_{1}} + \frac{\im z}{\lambda_{1} h}} \leq h^{- \delta}$ for $z \in \eqref{a12}$. Combining these estimates with \eqref{a14}, \eqref{a17} and \eqref{a32}, \eqref{a19} implies
\begin{align}
\big( 1 - h^{S / \lambda_{1} - 1 / 2} \widetilde{\CQ} \big)^{- 1} &= \big( 1 - h^{S / \lambda_{1} - 1 / 2} \CQ \big)^{- 1} \big( 1 - h^{S / \lambda_{1} - 1 / 2 + \delta} \CW \CQ \big)^{- 1} \big( 1 + \CO ( h^{\delta} ) \big)^{- 1}  \nonumber \\
&= \big( 1 + h^{S / \lambda_{1} - 1 / 2} \CQ \big) \big( 1 - h^{S / \lambda_{1} - 1 / 2 + \delta} \CW \CQ \big)^{- 1} + \CO ( 1 ) ,   \label{a29}
\end{align}
if $\dist (z , \res_{0} ( \widetilde{P} ) ) > \beta h \vert \ln h \vert^{- 1}$. Then \eqref{a18} follows.
\end{proof}

The next result provides a resonance free region for $P$ and the asymptotic of the resonances closest to the real axis for $\widetilde{P} = P + h^{1 + \delta} W$. Combined with Lemma \ref{a11}, it implies directly Theorem \ref{a1} with $D_{0} = \lambda_{2} / 2$ if we choose $\lambda_{1} = 1$.

\begin{lemma}\sl \label{a20}
There exists $\alpha > 0$ such that, for all $\delta > 0$ small enough and $C > 0$, the following properties hold for $h \in \SH$ small enough.

$i)$ $P$ has no resonance in
\begin{equation*}
E_{0} + [ - C h , C h ] + i \Big[ - \Big( \frac{\lambda_{2}}{2} + \alpha \Big) h , h \Big] .
\end{equation*}

$ii)$ In the domain \eqref{a12}, we have
\begin{equation*}
\dist \big( \res ( \widetilde{P} ) , \res_{0} ( \widetilde{P} ) \big) = o \Big( \frac {h}{\vert \ln h \vert} \Big) .
\end{equation*}
\end{lemma}

As in Definition 4.4 of \cite{BoFuRaZe18_01}, the notation $\dist ( A , B ) \leq \varepsilon$ in $C$ means that
\begin{align*}
&\forall a \in A \cap C , \quad \exists b \in B , \qquad \vert a - b \vert \leq \varepsilon ,   \\
\text{and} \quad &\forall b \in B \cap C , \quad \exists a \in A , \qquad \vert a - b \vert \leq \varepsilon .
\end{align*}
The proof of Lemma \ref{a20} gives a polynomial estimate of the resolvents in the corresponding domains (at distance larger than $h \vert \ln h \vert^{- 1}$ from the pseudo-resonances of $\widetilde{P}$).

\begin{proof}
The first point of the lemma has already been obtained in Lemma 12.1 of \cite{BoFuRaZe18_01}. In order to show the second point, we follow the strategy of Chapters 11 and 12 of \cite{BoFuRaZe18_01} and summarized in the introduction of \cite{BoFuRaZe18_01}. Then, we first prove that $\widetilde{P}$ has no resonance and we show a polynomial estimate of its resolvent away from the pseudo-resonances.

\begin{lemma}\sl \label{a21}
For $\delta > 0$ small enough, $C , \beta > 0$ and $h \in \SH$ small enough, $\widetilde{P}$ has no resonance in the domain
\begin{equation} \label{a22}
E_{0} + [ - C h , C h ] + i \Big[ - \Big( \frac{\lambda_{2}}{2} + \delta \lambda_{1} \Big) h - C \frac{h}{\vert \ln h \vert} , h \Big] \setminus \Big( \res_{0} ( \widetilde{P} ) + B \Big( 0 , \beta \frac{h}{\vert \ln h \vert} \Big) \Big) ,
\end{equation}
and there exists $M > 0$ such that the distorted operator $\widetilde{P}_{\theta}$ of angle $\theta = h \vert \ln h \vert$ satisfies
\begin{equation*}
\big\Vert( \widetilde{P}_{\theta} -z )^{-1} \big\Vert \lesssim h^{- M} ,
\end{equation*}
uniformly for $h \in \SH$ small enough and $z \in \eqref{a22}$.
\end{lemma}

\begin{figure}
\begin{center}
\begin{picture}(0,0)%
\includegraphics{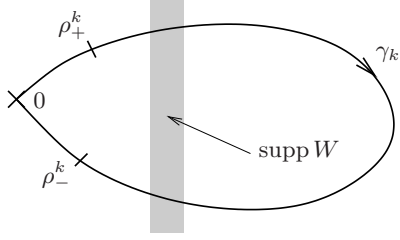}%
\end{picture}%
\setlength{\unitlength}{1381sp}%
\begingroup\makeatletter\ifx\SetFigFont\undefined%
\gdef\SetFigFont#1#2#3#4#5{%
  \reset@font\fontsize{#1}{#2pt}%
  \fontfamily{#3}\fontseries{#4}\fontshape{#5}%
  \selectfont}%
\fi\endgroup%
\begin{picture}(6990,4202)(3418,-6362)
\put(4351,-2761){\makebox(0,0)[lb]{\smash{{\SetFigFont{9}{10.8}{\rmdefault}{\mddefault}{\updefault}$\rho_{+}^{k}$}}}}
\put(3901,-4111){\makebox(0,0)[lb]{\smash{{\SetFigFont{9}{10.8}{\rmdefault}{\mddefault}{\updefault}$0$}}}}
\put(10051,-3211){\makebox(0,0)[lb]{\smash{{\SetFigFont{9}{10.8}{\rmdefault}{\mddefault}{\updefault}$\gamma_{k}$}}}}
\put(7951,-5011){\makebox(0,0)[lb]{\smash{{\SetFigFont{9}{10.8}{\rmdefault}{\mddefault}{\updefault}$\supp W$}}}}
\put(4051,-5461){\makebox(0,0)[lb]{\smash{{\SetFigFont{9}{10.8}{\rmdefault}{\mddefault}{\updefault}$\rho_{-}^{k}$}}}}
\end{picture}%
\end{center}
\caption{The geometric setting in the proof of Lemma \ref{a21}.} \label{f6}
\end{figure}

\begin{proof}[Proof of Lemma \ref{a21}]
This result is just an adaptation of Proposition 11.4 of \cite{BoFuRaZe18_01}. We only give the changes which have to be made in the present setting, sending back the reader to Section 11.2 of \cite{BoFuRaZe18_01} for the technical details. From the general arguments of Chapter 8 of \cite{BoFuRaZe18_01}, it is enough to show that any $u = u ( h ) \in L^{2} ( \R^{2} )$ and $z = z ( h ) \in \eqref{a22}$ with
\begin{equation} \label{a23}
\left\{ \begin{aligned}
&( \widetilde{P}_{\theta} - z ) u = \CO ( h^{\infty} ) ,    \\
&\Vert u \Vert_{L^{2} ( \R^{2} )} = 1 ,
\end{aligned} \right.
\end{equation}
vanishes microlocally near each point of $K ( E_{0} )$. For $k = 1 , 2 , 3$, let $u_{\pm}^{k}$ be microlocal restrictions of $u$ near $\rho_{\pm}^{k}$, where $\rho_{-}^{k}$ (resp. $\rho_{+}^{k}$) is a point on $\gamma_{k}$ just ``before'' (resp. ``after'') $ ( 0 , 0 )$ (see Figure \ref{f6}). As in \cite[(11.23)]{BoFuRaZe18_01}, they are Lagrangian distributions
\begin{equation*}
u_{-}^{k} \in \CI ( \Lambda_{+}^{1 , k} , h^{-N} ) \qquad \text{and} \qquad u_{+}^{k} \in \CI ( \Lambda_{+}^{0} , h^{-N} ) ,
\end{equation*}
associated to the Lagrangian manifold $\Lambda_{+}$ just after $( 0 , 0 )$ (denoted $\Lambda_{+}^{0}$) and after a turn along $\gamma_{k}$ (denoted $\Lambda_{+}^{1 , k}$) for some $N \in \R$.

After an appropriate renormalization (see \cite[(11.25)]{BoFuRaZe18_01}), the symbols $a_{-}^{k} ( x , h ) \in S ( h^{- N} )$ of $u_{-}^{k}$ satisfy the relation
\begin{equation} \label{a25}
a_{-}^{k} ( x , h ) = h^{S ( z , h ) / \lambda_{1} - 1 / 2} \sum_{\ell = 1}^{3} \CP_{k , \ell} ( x , h ) a_{-}^{\ell} ( x_{-}^{\ell} , h ) + S ( h^{- N + \zeta - \delta} ) ,
\end{equation}
near $x_{-}^{k} = \pi_{x} ( \rho_{-}^{k} )$. In this expression, the symbols $\CP_{k , \ell} \in S (1)$ (resp. the constant $\zeta > 0$) are independent of $u$ (resp. $\delta , u$) and $\CP_{k , \ell} ( x_{-}^{k} , h ) = \widetilde{\CQ}_{k , \ell} ( z , h )$. Compared with \cite[(11.27)]{BoFuRaZe18_01}, $\CQ$ is replaced by $\widetilde{\CQ}$ in $\CP_{k , \ell} ( x_{-}^{k} , h )$. Indeed, no change has to be made for the propagation through the fixed point $( 0 , 0 )$ since $W$ is supported away from $V_{\rm top}$ (see \cite[(11.29)]{BoFuRaZe18_01}), but the usual transport equation near $\gamma_{k}$
\begin{equation*}
2 \nabla \varphi_{+} \cdot \nabla a_{-}^{k} + ( \Delta \varphi_{+} - i \sigma ) a_{-}^{k} = \CO ( h^{- N + 1} ) ,
\end{equation*}
with $\sigma = ( z - E_{0} ) / h$ is replaced by
\begin{equation*}
2 \nabla \varphi_{+} \cdot \nabla a_{-}^{k} + ( \Delta \varphi_{+} - i \sigma + i h^{\delta} W ) a_{-}^{k} = \CO ( h^{- N + 1} ) ,
\end{equation*}
giving on the curve $\gamma_{k}$
\begin{equation*}
\partial_{t} a_{-}^{k} ( x_{k} ( t ) ) + ( \Delta \varphi_{+} - i \sigma + i h^{\delta} W ) a_{-}^{k} ( x_{k} ( t ) ) = \CO ( h^{- N + 1} ) ,
\end{equation*}
and leading to the additional factor $e^{- i h^{\delta} \int W ( x_{k} ( t ) ) \, d t}$ in the quantization matrix $\widetilde{\CQ}$ (see \cite[(11.31)]{BoFuRaZe18_01}). On the other hand, the remainder term $\CO ( h^{- N + \zeta - \delta} )$ in \eqref{a25} comes from the fact that $\vert h^{S / \lambda_{1} - 1 / 2} \vert \lesssim h^{- \delta}$ uniformly for $z \in \eqref{a22}$ (see Chapter 12.2 of \cite{BoFuRaZe18_01} for a similar argument).

Applying \eqref{a25} with $x = x_{-}^{k}$, we get
\begin{equation*}
\big( 1 - h^{S ( z , h ) / \lambda_{1} - 1 / 2} \widetilde{\CQ} ( z , h ) \big) a_{-} ( x_{-} , h ) = \CO ( h^{- N + \zeta - \delta} ) ,
\end{equation*}
where $a_{-} ( x_{-} , h )$ is a shortcut for the $3$-vector with coefficients $a_{-}^{k} ( x_{-}^{k} , h )$. From \eqref{a18}, it yields
\begin{equation*}
\vert a_{-} ( x_{-} , h ) \vert \lesssim h^{- N + \zeta - 2 \delta} ,
\end{equation*}
uniformly for $z \in \eqref{a22}$. Using again \eqref{a25}, we deduce $a_{-}^{k} \in S ( h^{- N + \zeta - 3 \delta} ) \subset S ( h^{- N + \zeta /2} ) $ for $\delta > 0$ small enough. Thus, starting from $a_{-}^{k} \in S ( h^{- N} )$, we have proved $a_{-}^{k} \in S ( h^{- N + \zeta / 2} )$. By a bootstrap argument (see the end of Chapter 9 of \cite{BoFuRaZe18_01}), we obtain $u = \CO ( h^{\infty} )$ microlocally near $K ( E_{0} )$ and the lemma follows.
\end{proof}

To finish the proof of Lemma \ref{a20}, it remains to show that $\widetilde{P}$ has a resonance near each pseudo-resonance. That is

\begin{lemma}\sl \label{a26}
For $\delta > 0$ small enough, $C , \beta > 0$ and $h \in \SH$ small enough, the operator $\widetilde{P}$ has at least one resonance in $B ( z , \beta h \vert \ln h \vert^{- 1} )$ for any pseudo-resonance $z \in \eqref{a12}$.
\end{lemma}

\begin{proof}[Proof of Lemma \ref{a26}]
This result is equivalent to Proposition 11.6 of \cite{BoFuRaZe18_01} in the present setting, and we only explain how to adapt its proof. If Lemma \ref{a26} did not hold, there would exist a sequence $z = z( h ) \in \eqref{a12}$ of pseudo-resonances where $h \in \SH$ goes to $0$ such that
\begin{equation} \label{a27}
\widetilde{P} \text{ has no resonance in } \CD = B \Big( z , \beta \frac{h}{\vert \ln h \vert} \Big) .
\end{equation}

We now construct a ``test function''. Let $\widetilde{v}$ be a WKB solution near $x_{-} ^{1}$ of
\begin{equation*}
\left\{ \begin{aligned}
&( \widetilde{P} - \widetilde{z} ) \widetilde{v} = 0 &&\text{near } x_{-}^{1} ,   \\
&\widetilde{v} (x) = e^{i \varphi_{+}^{1 , 1} ( x ) / h} &&\text{on } \vert x \vert = \vert x_{-}^{1} \vert \text{ near } x_{-}^{1} ,
\end{aligned} \right.
\end{equation*}
for $\widetilde{z} \in \partial \CD$, where $\varphi_{+}^{1 , 1}$ is a generating function of $\Lambda_{+} ^{1 , 1}$. Note that $\widetilde{P} = P$ near $x_{-}^{1}$ and that $\widetilde{v}$ can be chosen holomorphic with respect to $\widetilde{z}$ near $\CD$. After multiplication by a renormalization factor as in \cite[(11.44)]{BoFuRaZe18_01}, this function is denoted $\widehat{v}$. Consider cut-off functions $\chi , \psi \in C^{\infty}_{0} ( T^{*} \R^{2} )$ such that $\chi = 1$ near $\rho_{-}^{1}$ and $\psi = 1$ near the part of the curve $\supp ( \nabla \chi ) \cap \gamma_{1}$ before $\rho_{-}^{1}$. Then, we take as ``test function''
\begin{equation*}
v = \Op ( \psi ) \big[ \widetilde{P} , \Op ( \chi ) \big] \widehat{v} .
\end{equation*}

Let $u \in L^{2} ( \R^{2} )$ be the solution of
\begin{equation} \label{a31}
( \widetilde{P}_{\theta} - \widetilde{z} ) u = v ,
\end{equation}
for $\widetilde{z} \in \partial \CD$. From Lemmas \ref{a11} and \ref{a21}, $u$ is well-defined and polynomially bounded. Let $u_{-}^{k}$ be a microlocal restriction of $u$ near $\rho_{-}^{k}$ as before. Working as in Lemma 11.10 of \cite{BoFuRaZe18_01}, one can show that $u_{-}^{k} \in \CI ( \Lambda_{+}^{1 , k} , h^{- 2 \delta} )$ with renormalized symbol $a_{-}^{k}$. Moreover, as in \eqref{a25}, we get
\begin{equation} \label{a24}
a_{-}^{k} ( x , h ) = h^{S ( z , h ) / \lambda_{1} - 1 / 2} \sum_{\ell = 1}^{3} \CP_{k , \ell} ( x , h ) a_{-}^{\ell} ( x_{-}^{\ell} , h ) + \widetilde{a}_{k} ( x , h ) + S ( h^{\zeta - 3 \delta} ) ,
\end{equation}
near $x_{-}^{k}$, where $\widetilde{a}_{k}$ denotes the symbol of $\widetilde{v}$ near $x_{-}^{k}$. In particular, $\widetilde{a}_{k} ( x , h ) = 0$ for $k \neq 1$ and $\widetilde{a}_{1} ( x_{-}^{1} , h ) = 1$. This relation is obtained using the proofs of \eqref{a25} and Lemma 11.8 of \cite{BoFuRaZe18_01}.

To obtain a contradiction with \eqref{a27}, we consider
\begin{equation} \label{a28}
\SI = \frac{1}{2 i \pi} \int_{\partial \CD} u ( \widetilde{z} ) \, d \widetilde{z} .
\end{equation}
From the properties of $u_{-}^{k}$ and $\vert \partial \CD \vert = 2 \pi \beta h \vert \ln h \vert^{- 1}$, we have $\SI \in \CI ( \Lambda_{+}^{1 , k} , h^{1 - 2 \delta} \vert \ln h \vert^{- 1} )$ microlocally near $\rho_{-}^{k}$, where its renormalized symbol $b_{k} ( x , h )$ satisfies
\begin{equation} \label{a30}
b_{k} ( x , h ) = \frac{1}{2 i \pi} \int_{\partial \CD} a_{-}^{k} ( x , h ) \, d \widetilde{z} .
\end{equation}
Applying \eqref{a24} with $x = x_{-}^{k}$ leads to
\begin{equation*}
\big( 1 - h^{S ( z , h ) / \lambda_{1} - 1 / 2} \widetilde{\CQ} ( z , h ) \big) a_{-} ( x_{-} , h ) = \widetilde{a} ( x_{-} , h ) + \CO ( h^{\zeta - 3 \delta} ) ,
\end{equation*}
where $c ( x_{-} , h )$ is a generic shortcut for the $3$-vector with coefficients $c^{k} ( x_{-}^{k} , h )$. It implies
\begin{align*}
a_{-} ( x_{-} , h ) &= \big( 1 - h^{S / \lambda_{1} - 1 / 2} \widetilde{\CQ} \big)^{- 1} \widetilde{a} ( x_{-} , h ) + \CO ( h^{\zeta - 4 \delta} )   \\
&= \big( 1 + h^{S / \lambda_{1} - 1 / 2} \CQ \big) \big( 1 - h^{S / \lambda_{1} - 1 / 2 + \delta} \CW \CQ \big)^{- 1} \widetilde{a} ( x_{-} , h ) + \CO ( 1 ) + \CO ( h^{\zeta - 4 \delta} ) .
\end{align*}
from \eqref{a18} and \eqref{a29}. We deduce
\begin{align*}
\CW a_{-} ( x_{-} , h ) ={}& \CW \big( 1 - h^{S / \lambda_{1} - 1 / 2 + \delta} \CW \CQ \big)^{- 1} \widetilde{a} ( x_{-} , h ) - h^{- \delta} \widetilde{a} ( x_{-} , h )   \\
&+ h^{- \delta} \big( 1 - h^{S / \lambda_{1} - 1 / 2 + \delta} \CW \CQ \big)^{- 1} \widetilde{a} ( x_{-} , h ) + \CO ( 1 ) + \CO ( h^{\zeta - 4 \delta} ) .
\end{align*}
Inserting this expression in \eqref{a30} and using \eqref{a14} yield
\begin{equation*}
\CW b ( x_{-} , h ) = \frac{h^{- \delta}}{2 i \pi} \int_{\partial \CD} \big( 1 - h^{S / \lambda_{1} - 1 / 2 + \delta} \CW \CQ \big)^{- 1} \widetilde{a} ( x_{-} , h ) \, d \widetilde{z} + \CO \Big( \frac{h}{\vert \ln h \vert} \Big) + \CO \Big( \frac{h^{1 + \zeta - 4 \delta}}{\vert \ln h \vert} \Big) .
\end{equation*}
Note that $\widetilde{a} ( x_{-} , h ) = {}^{t} ( 1 , 0 , 0 )$ is an explicit eigenvector of $\CW \CQ$ associated to its non-zero eigenvalue $- i w \CQ_{1 , 1} ( z , h )$ (see \eqref{a9}). Thus, computing the integral as in \cite[(11.67)]{BoFuRaZe18_01}, we get
\begin{equation*}
\CW b ( x_{-} , h ) = i \lambda_{1} \frac{h^{1 - \delta}}{\vert \ln h \vert} \widetilde{a} ( x_{-} , h ) + o \Big( \frac{h^{1 - \delta}}{\vert \ln h \vert} \Big) + \CO \Big( \frac{h}{\vert \ln h \vert} \Big) + \CO \Big( \frac{h^{1 + \zeta - 4 \delta}}{\vert \ln h \vert} \Big) .
\end{equation*}
Taking $\delta > 0$ small enough and using that $\CW b \in S ( h^{1 - 2 \delta} \vert \ln h \vert^{- 1} )$, the previous asymptotic shows that $\CW b \neq 0$ so that $\SI \neq 0$. On the other hand, since $\widetilde{P}$ has no resonance in $\CD$ (see \eqref{a27}), the function $u$ defined by \eqref{a31} is holomorphic in $\CD$ and \eqref{a28} gives $\SI = 0$. Eventually, we get a contradiction and the lemma follows.
\end{proof}

The second point of Lemma \ref{a20} is a direct consequence of Lemmas \ref{a21} and \ref{a26}.
\end{proof}

\bibliographystyle{amsplain}


\begin{thebibliography}{10}

\bibitem{Ag98_01}
S.~Agmon, \emph{A perturbation theory of resonances}, Comm. Pure Appl. Math.
  \textbf{51} (1998), no.~11-12, 1255--1309.

\bibitem{Ag98_02}
S.~Agmon, \emph{Erratum: ``{A} perturbation theory of resonances''}, Comm. Pure
  Appl. Math. \textbf{52} (1999), no.~12, 1617--1618.

\bibitem{BoBuRa10_01}
J.-F Bony, N.~Burq, and T.~Ramond, \emph{Minoration de la r\'{e}solvante dans
  le cas captif}, C. R. Math. Acad. Sci. Paris \textbf{348} (2010), no.~23-24,
  1279--1282.

\bibitem{BoFuRaZe18_01}
J.-F. Bony, S.~Fujii\'e, T.~Ramond, and M.~Zerzeri, \emph{Resonances for
  homoclinic trapped sets}, Ast\'erisque (2018), no.~405, vii+314.

\bibitem{DyWa16_01}
S.~Dyatlov and A.~Waters, \emph{Lower resolvent bounds and {L}yapunov
  exponents}, Appl. Math. Res. Express. AMRX (2016), no.~1, 68--97.

\bibitem{DyZw16_01}
S.~Dyatlov and M.~Zworski, \emph{Mathematical theory of scattering resonances},
  Graduate Studies in Mathematics, vol. 200, American Mathematical Society,
  2019.

\bibitem{FuLaMa11_01}
S.~Fujii{\'e}, A.~Lahmar-Benbernou, and A.~Martinez, \emph{Width of shape
  resonances for non globally analytic potentials}, J. Math. Soc. Japan
  \textbf{63} (2011), no.~1, 1--78.

\bibitem{Ge88_01}
C.~G{\'e}rard, \emph{Asymptotique des p\^oles de la matrice de scattering pour
  deux obstacles strictement convexes}, M\'em. Soc. Math. France (1988),
  no.~31, 146.

\bibitem{GeSj87_01}
C.~G{\'e}rard and J.~Sj{\"o}strand, \emph{Semiclassical resonances generated by
  a closed trajectory of hyperbolic type}, Comm. Math. Phys. \textbf{108}
  (1987), 391--421.

\bibitem{HeSj86_01}
B.~Helffer and J.~Sj{\"o}strand, \emph{R\'esonances en limite semi-classique},
  M\'em. Soc. Math. France (1986), no.~24-25, iv+228.

\bibitem{Ik83_01}
M.~Ikawa, \emph{On the poles of the scattering matrix for two strictly convex
  obstacles}, J. Math. Kyoto Univ. \textbf{23} (1983), no.~1, 127--194.

\bibitem{NoZw09_01}
S.~Nonnenmacher and M.~Zworski, \emph{Quantum decay rates in chaotic
  scattering}, Acta Math. \textbf{203} (2009), no.~2, 149--233.

\bibitem{Sj87_01}
J.~Sj{\"o}strand, \emph{Semiclassical resonances generated by nondegenerate
  critical points}, Pseudodifferential operators (Oberwolfach, 1986), Lecture
  Notes in Math., vol. 1256, Springer, 1987, pp.~402--429.

\bibitem{Sj07_01}
J.~Sj{\"o}strand, \emph{Lectures on resonances}, preprint at {\tt
  http://sjostrand.perso.math.cnrs.fr/} (2007), 1--169.

\bibitem{TrEm05_01}
L.~Trefethen and M.~Embree, \emph{Spectra and pseudospectra}, Princeton
  University Press, 2005, The behavior of nonnormal matrices and operators.

\end{thebibliography}

\end{document}